\documentclass[12pt]{amsart}

\usepackage[T1]{fontenc}
\usepackage[utf8]{inputenc}
\usepackage{mathptmx}
\usepackage{listings}
\usepackage{graphicx}
\usepackage{setspace}
\title[On the zeros of the Pearcey integral]{On the zeros of the Pearcey integral and a Rayleigh-type equation}
\author{Gerardo Hern\'andez-del-Valle}
\address{Direcci\'on General de Investigaci\'on Econ\'omica, Banco de M\'exico, Av. 5 de Mayo \# 18, Col. Centro Hist\'orico, M\'exico, D.F. CP 06059}
\email{gerardo.hernandez@banxico.org.mx}
\date{November 18, 2014}

\newtheorem{theorem}{Theorem}[section]
\newtheorem{example}[theorem]{Example}
\newtheorem{examplen}[theorem]{Numerical Example}
\newtheorem{remark}[theorem]{Remark}
\newtheorem{definition}[theorem]{Definition}
\newtheorem{algorithm}[theorem]{Algorithm}
\begin{document}
\maketitle
\begin{abstract}
In this work we find a sequence of functions $f_n$ at which the integral
\begin{eqnarray}\label{zero}
v(t,x)=\int_{-\infty}^{\infty}e^{i\lambda x-\lambda^2t/2-\lambda^4/4}d\lambda
\end{eqnarray}
 is identically zero for all $t\geq 0$, that is
\begin{eqnarray*}
v(t,f_n(t))=0\qquad \forall t\geq 0.
\end{eqnarray*}
The function $v$, after  proper change of variables and rotation of the path of integration,  is known as the Pearcey integral or Pearcey function, indistinctly. We also show that each  $f_n$ is expressed in terms of a second order non-linear ODE, which turns out to be of the Rayleigh-type. Furthermore the initial conditions, which uniquely determine each $f_n$, depend on the zeros of an Airy function of order 4  defined as
\begin{eqnarray*}
\phi(x)=\int_{-\infty}^{\infty}e^{i\lambda x-\lambda^4/4}d\lambda.
\end{eqnarray*}

As a byproduct of these facts, we develop a methodology to find a class of functions  which solve the moving boundary problem of the heat equation. To this end, we make use of generalized Airy functions, which in some particular cases fall within the category of functions with infinitely many real zeros, studied by P\'olya.
\end{abstract}

\section{Introduction}
The Pearcey integral was first evaluated numerically by Pearcey  \cite{pearcey} in his investigation of the electromagnetic field near a cusp. The integral appears also  in optics \cite{berry}, in the asymptotics of special functions \cite{kaminski1}, in  probability theory \cite{tracy}, as the generating function of heat (and hence Hermite) polynomials of order $4k$ for $k\in \mathbb{N}$ \cite{widder}.  It also  falls into the category of functions considered by P\'olya \cite{polya}, that is functions with countably many zeros. For the numerical evaluation of the zeros of the Pearcey integral see for instance \cite{kaminski1}, this will be important since this zeros will correspond to the initial value of each function $f_n(0)$.

The main motivation  in finding the zeros of the Pearcey function, which solves the heat equation
\begin{eqnarray}\label{heat}
h_t(t,x)=\frac{1}{2}h_{xx}(t,x),
\end{eqnarray}
is due to the fact  that the  main building block used to construct  the density of the first time that a Wiener process hits a boundary $f$, is to find a function  $f$ such that
\begin{eqnarray*}
h(t,f(t))=0\qquad\forall t\geq 0.
\end{eqnarray*}
For example, suppose there is a financial contract which will be activated if ever the price of an asset $S$ (modelled as Brownian motion) reaches a prescribed boundary $f$. For instance, in Figure \ref{fig007} the blue line represents the evolution of the price of  $S_t$, for $t\in[0,10]$, in turn the red line represents a boundary which activates a contract if it is ever reached. In particular, the barrier option is a  contract of this type. For a more detailed exposition see for instance \cite{hernandez}. 
\begin{figure}
\hspace{-.5 cm}\includegraphics[scale=.5,height=7cm,width=9cm]{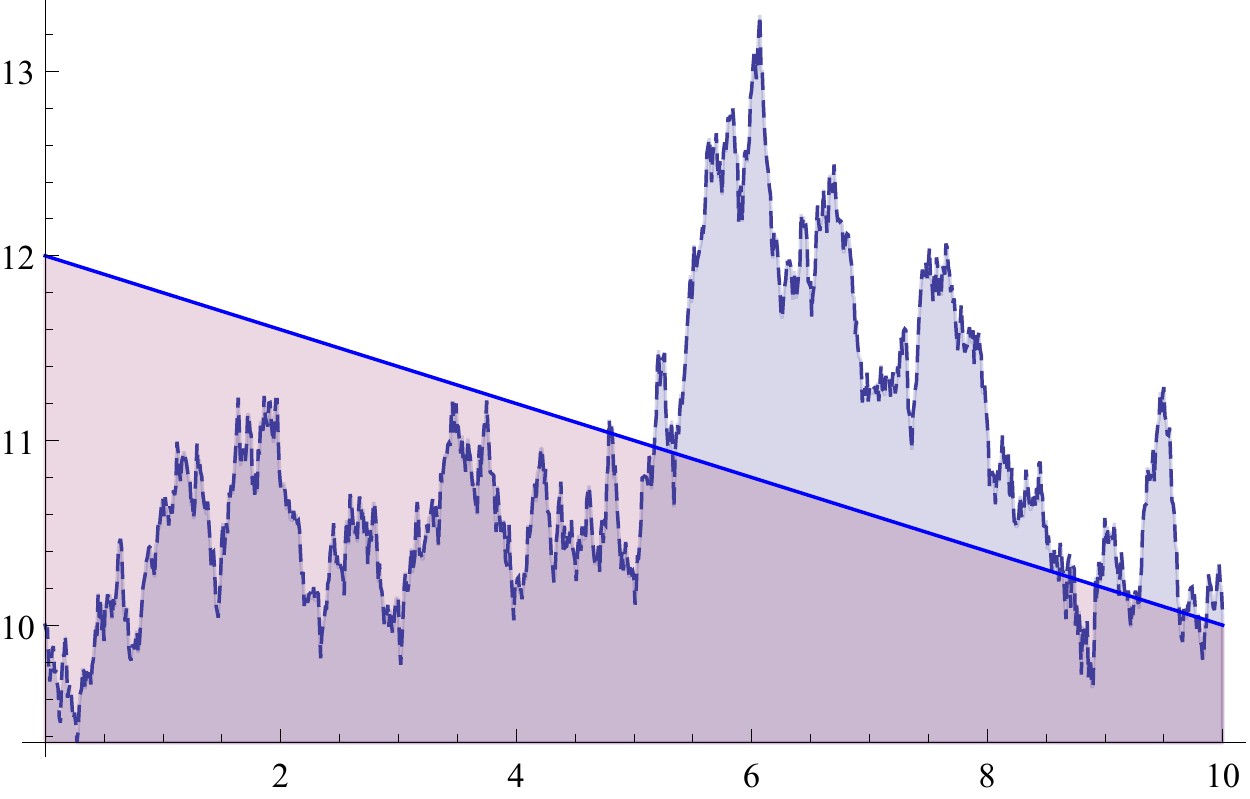}
\caption{The dotted line represents the price of some asset $S$, in turn, the solid line represents a boundary which will activate a contract the first time it is reached. In this example the contract was activated approximately at time  $t=5$. This is a random time, since it would have been impossible to foresee the outcome.\label{fig007}  }   
\end{figure}

Next, we note that for some constant $b\in\mathbb{R}$, the  function
\begin{eqnarray}\label{linear1}
h(t,x)=\frac{x}{\sqrt{2\pi t^3}}\exp\left\{-\frac{x^2}{2t}\right\}+b\frac{1}{\sqrt{2\pi t}}\exp\left\{-\frac{x^2}{2t}\right\},
\end{eqnarray}
with $(t,x)\in\mathbb{R}^+\times\mathbb{R}$, solves the heat equation (\ref{heat}). This is true since $h$ in (\ref{linear1}) is a linear combination of the fundamental solution of (\ref{heat}) and  its first derivative with respect to the space variable $x$. It is clear that the function (\ref{linear1})  equals  zero at $x=-bt$. Hence, for any $a\in\mathbb{R}$, and setting $x=a-bt$ in (\ref{linear1})  we obtain
$$
h(t,a-bt)=\frac{a}{\sqrt{2\pi t^3}}\exp\left\{-\frac{(a-bt)^2}{2t}\right\}, \quad(t,x)\in\mathbb{R}^+\times\mathbb{R}.
$$
We note that the right-hand side of this identity is in fact the density of the first time that a Brownian motion hits a linear boundary \cite[p. 196]{karatzas}. In practice these results are used for instance in (a) the valuation problem of financial assets, in particular in the valuation of {\it barrier options\/} [see Bj\"ork (2009)],  (b) in the quantificaction of counterparty risk [see Davis and Pistorious (2010)], and in general in physical problems.

The  main contributions of this work are, on the one hand,  finding the zeros of the Pearcey integral. On the other, advancing in the direction of developing a rather simple and straightforward methodology  to find explicit solutions of the time-varying boundary problem for the heat equation. In this regard we note that there exist  techniques to study the latter aforementioned problem in terms of solutions to integral equations \cite{fokas}. We recall that solutions in terms of integral equations,  in general can only be evaluated numerically. In turn, our approach leads to solutions in terms of ODEs.

The paper is organized as follows. In Section \ref{secairy} we introduce the Airy function of order 4. Next in Section \ref{secpearcey} we define the Pearcey integral and describe its connection with the Airy function of order 4. In Section \ref{secrayleigh} we derive a Rayleigh-type equation, whose solution  kills the Pearcey function. The techniques described in Section \ref{secrayleigh} are illustrated with examples in Section \ref{secexamples}. In Section \ref{secasymptotic} we derive a function which asymptotically solves the moving boundary problem for the Pearcey integral. This approximation can be helpful in the numerical solution of the Rayleigh equation. We conclude in Section \ref{conclusions}, with some final remarks.

\section{Generalized Airy function of order 4}\label{secairy}
With respect to the zeros of Fourier integrals, P\'olya proved  \cite{polya}  that all the zeros of
\begin{eqnarray}\label{poleq}
\int_{-\infty}^\infty e^{-u^{2m}+izu}du,\qquad \hbox{for }m=1,2, 3\dots
\end{eqnarray}
are real and infinitely many for $m>1$. In turn, the generalized Airy function $\phi$  of order 4 can be expressed as a solution of the following ODE
\begin{eqnarray}\label{ai}
\phi^{(3)}&=&x\phi,\\
\nonumber\phi^{(j)}&=&(j-3)\phi^{(j-4)}+x\phi^{(j-3)}\qquad \hbox{for } j>3.
\end{eqnarray}
One can prove, for instance applying the Fourier transform to (\ref{ai}) and solving the resulting equation,  that $\phi$  is a particular case of  (\ref{poleq})  when $m=2$, namely
\begin{eqnarray}\label{airy4}
\phi(x)=\frac{1}{2\pi}\int_{-\infty}^{\infty}\exp\left\{ixy-\frac{y^4}{4}\right\}dy.
\end{eqnarray}
Furthermore function $\phi$ is symmetric, with countably many zeros in the real line---hence oscillatory---and tends to zero as it increases to $\pm\infty$, see Figure \ref{fig00}.
Regarding the zeros of (\ref{airy4}), there exist asymptotic estimates which are  derived by means of the method of steepest descent \cite{senouf}.

\begin{figure}
\hspace{-.5 cm}\includegraphics[scale=.5,height=7cm,width=9cm]{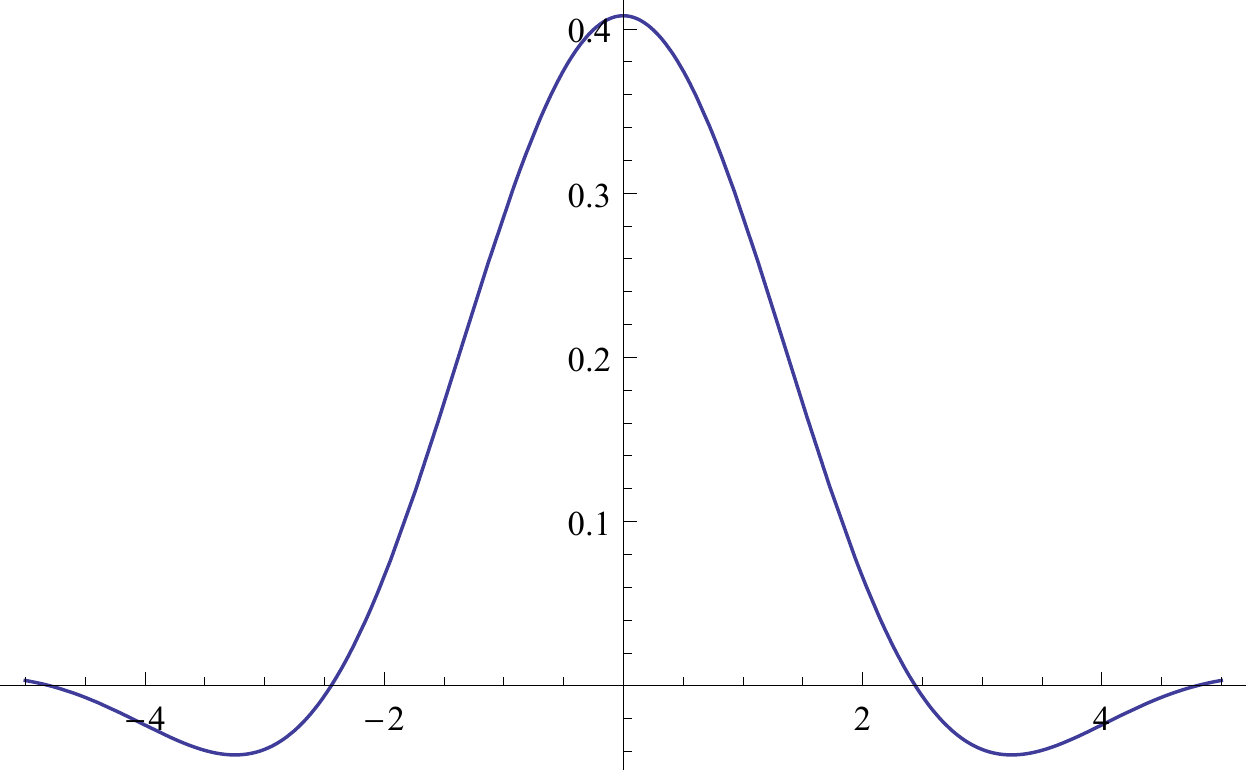}
\caption{We plot function $\phi$, defined in (\ref{airy4}). We observe the function is even and its first zero is at $\pm 2.44197$.}   \label{fig00} 
\end{figure}
\section{The Pearcey integral}\label{secpearcey}
Paris \cite{paris91,paris} analyzed  the asymptotic behavior of

\begin{eqnarray*}
P_n'(X,Y)=\int_{-\infty}^{\infty}e^{i(u^{2n}+Xu^n+Yu)}du,\qquad n\in\mathbb{N},\ n\geq 2,
\end{eqnarray*}
which  by rotation of the path of integration $(u=te^{\frac{\pi i}{4n}})$ and use of Jordan's lemma  (see \cite{senouf}) can be expressed as
\begin{eqnarray*}
P_n'(X,Y)=P_n(x,y)=e^{\frac{\pi i}{4n}}\int_{-\infty}^{\infty}e^{-t^{2n}-xt^n+iyt}dt,
\end{eqnarray*}
with $x=Xe^{-\frac{\pi i}{4}}$ and $y=Ye^{\frac{\pi i}{4n}}$. In particular, 
the Pearcey integral, which solves (\ref{heat}),  is the case $P_2$. More explicitly, we have the following.
\begin{definition} \cite{pearcey}The Pearcey integral  is defined as
\begin{eqnarray}\label{p2}
P_2'(X,Y)=\int_{-\infty}^{\infty}e^{i(u^{4}+Xu^2+Yu)}du.
\end{eqnarray}

\end{definition}
 In this work we study instead the  following Fourier integral
\begin{eqnarray}\label{pearcey}
v(t,x):=\frac{1}{2\pi}\int_{-\infty}^{\infty}\exp\left\{i\lambda x-\frac{1}{2}\lambda^2t-\frac{\lambda^4}{4}\right\}d\lambda,
\end{eqnarray}
because the zeros of $v$, for $(t,x)\in\mathbb{R}^+\times\mathbb{R}$, are expressed in terms of a continuously differentiable function $f$, as opposed to (\ref{p2}). See \cite{senouf}.
\begin{remark}
We observe that the  function $v$ in (\ref{pearcey})   is the convolution between the kernel of standard Brownian motion and the generalized Airy function of order 4, in equation (\ref{ai}).
\end{remark}

\section{Zeros of the Pearcey function}\label{secrayleigh}
\begin{remark}
Throughout this work, the $n$-th partial differentiation with respect to the space variable $x$ of any given function $v(t,x)$ is denoted as $v^{(n)}$.
\end{remark}
In this section we find the function $f$ for which the Pearcey function is zero for every $t\geq 0$. The idea is to exploit, on the one hand, the differential form of the Airy function of order 4, defined in (\ref{ai}), and on the other to use the fact that the Pearcey function solves the heat equation (\ref{heat}).

The main result is the following.
\begin{theorem}\label{th1} Suppose that $v$  is as in (\ref{pearcey}), $\phi$ solves (\ref{ai}),  $\xi$ is  such that $\phi(\xi)=0$, and
$f$ is a solution to the following Rayleigh-type ODE
\begin{eqnarray}\label{odesol}
f''(t)=2\left[f'(t)\right]^3-\frac{1}{2}tf'(t)-\frac{1}{4}f(t),
\end{eqnarray}
with $f(0)=\xi$ and $f'(0)=-\phi^{(2)}(\xi)/[2\phi^{(1)}(\xi)]$. Then, for every $t\geq 0$, we have 
\begin{eqnarray*}
v(t,f(t))=0.
\end{eqnarray*}
\end{theorem}
\begin{proof} Given that $\phi$ is as in (\ref{airy4}), its Fourier transform equals
\begin{eqnarray*}
\tilde{\phi}(\lambda)=\exp\left\{-\frac{\lambda^4}{4}\right\}.
\end{eqnarray*}
Furthermore, if we apply the Fourier transform directly to the ODE (\ref{ai}) we have that
\begin{eqnarray*}
(i\lambda)^3\tilde{\phi}=i\frac{d}{d\lambda}\tilde{\phi}.
\end{eqnarray*}
Since this expression is already in Fourier domain, we convolve the previous expression with the heat kernel as follows
\begin{eqnarray*}
\int e^{i\lambda x-\frac{1}{2}\lambda^2t}(i\lambda)^3\tilde{\phi}d\lambda&=&\int ie^{i\lambda x-\frac{1}{2}\lambda^2t}\frac{d}{d\lambda}\tilde{\phi}d\lambda\\
&=&\int ie^{i\lambda x-\frac{1}{2}\lambda^2t}d\tilde{\phi}.
\end{eqnarray*}
This in turn, and by direct application of the integration by parts formula,  yields
\begin{eqnarray}\label{ai41}
\nonumber v^{(3)}&=&-\int\tilde{\phi}i(ix-\lambda t)e^{i\lambda x-\frac{1}{2}\lambda^2t}d\lambda,\\
&=&xv+tv^{(1)},
\end{eqnarray}
as well as 
\begin{eqnarray}\label{ai42}
v^{(4)}&=&v+xv^{(1)}+tv^{(2)},
\end{eqnarray}
after differentiation with respect to $x$.

Next given that there exists an $f$, see P\'olya \cite{polya}, such that the following holds for all $t$
\begin{eqnarray*}
v(t,f(t))=0,
\end{eqnarray*}
we differentiate (Leibniz integral rule) $v$ with $x=f(t)$, defined in (\ref{pearcey}), with respect to $t$ to obtain
\begin{eqnarray*}
\frac{1}{2\pi}\int_{-\infty}^{\infty}(i\lambda f'(t)-1/2\lambda^2)\exp\left\{i\lambda f(t)-\frac{1}{2}\lambda^2t-\frac{\lambda^4}{4}\right\}d\lambda=0,
\end{eqnarray*}
which is equivalent to
\begin{eqnarray}\label{hit}
f'(t)v^{(1)}(t,f(t))+\frac{1}{2}v^{(2)}(t,f(t))&=&0
\end{eqnarray}
and
\begin{eqnarray}\label{hit2}
 f''(t)v^{(1)}+f'(t)(f'(t)v^{(2)}+v^{(3)})+\frac{1}{4}v^{(4)}&=&0
\end{eqnarray}
after differentiation with respect to $t$ twice. We note that equations (\ref{ai41}) and (\ref{ai42}) obtained from the Airy differential equation (\ref{ai}) , as well as (\ref{hit}) and (\ref{hit2}) obtained from the heat equation, involve derivatives of $v$ up to order 4. What remains is to obtain (\ref{odesol}) from these expressions. To this end, from  (\ref{ai41}) and (\ref{ai42}) we first have
\begin{eqnarray*}
\frac{v^{(3)}(t,f(t))}{v^{(1)}(t,f(t))}&=&t\\
\frac{v^{(4)}(t,f(t))}{v^{(1)}(t,f(t))}&=&f(t)+t\frac{v^{(2)}(t,f(t))}{v^{(1)}(t,f(t))}.
\end{eqnarray*}
Next from (\ref{hit}) and (\ref{hit2}) it follows that
\begin{eqnarray*}
f'(t)+\frac{1}{2}\frac{v^{(2)}(t,f(t))}{v^{(1)}(t,f(t))}=0\\
f''(t)+f'(t)\left(f'(t)\frac{v^{(2)}(t,f(t))}{v^{(1)}(t,f(t))}+\frac{v^{(3)}(t,f(t)}{v^{(1)}(t,f(t))}\right)=-\frac{1}{4}\frac{v^{(4)}(t,f(t))}{v^{(1)}(t,f(t))}.
\end{eqnarray*}
These identities yield
\begin{eqnarray*}
f''(t)+f'(t)\left(-2(f'(t))^2+t\right)=-\frac{1}{4}(f(t)-2tf'(t)).
\end{eqnarray*}
This completes the proof of Theorem \ref{th1}.
\end{proof}
\section{Examples}\label{secexamples}
To illustrate Theorem \ref{th1} we next present examples and numerical experiments.
\begin{examplen}\label{hx} One can show that at $t=0$  the following two identities hold for $f$: $f(0)=\xi=2.44197$ and $f'(0)=0.729925$. Hence, from Theorem \ref{th1}, we may plot the solution of (\ref{odesol}),  see Figure \ref{figa}, for $t\in[0,4]$. The ODE was solved numerically using Mathematica.
\end{examplen}

{\tiny
\begin{lstlisting}
Airy4[x_] := 
 1/(2*Pi)*NIntegrate[Exp[I*x*y - y^4/4], {y, -Infinity, Infinity}]
Airy41[x_] := 
 NIntegrate[I Exp[I x y - y^4/4]*y, {y, -Infinity, Infinity}]/(
 2 \[Pi])
Airy42[x_] := 
 NIntegrate[-Exp[I x y - y^4/4]*y^2, {y, -Infinity, Infinity}]/(
 2 \[Pi])
x1 = FindRoot[Re[Airy4[x]], {x, 2.44}]
x1 = x /. x1
fp = -N[Re[Airy42[N[x1]]]]/(2*Re[Airy41[N[x1]]])
s = NDSolve[{g''[x] == -(g[x]/4) - 
     1/2 x g'[x] + 2 (g'[x])^3, g[0] == x1, 
  g'[0] == fp}, g, {x, 0, 4}, AccuracyGoal -> 20, 
  PrecisionGoal -> 10, WorkingPrecision -> 33]
Plot[Evaluate[g[x] /. s], {x, 0, 4}, PlotRange -> All]
\end{lstlisting}
}

\begin{examplen} To test the accuracy of the solution in Example \ref{hx} we may use the following code in Mathematica in the interval $t\in[0,4]$.
\end{examplen}
{\tiny
\begin{lstlisting}
test[x_] := g[x] /. s
F0[t_, x_] := 
 1/(2*Pi)*NIntegrate[
   Exp[I*x*y - y^2*t/2 - y^4/4], {y, -Infinity, Infinity}]
Table[Re[F0[i/100, test[i/100][[1]]]], {i, 0, 400, 1}]
\end{lstlisting}

}

\begin{figure}
\hspace{-.5 cm}\includegraphics[scale=.5,height=7cm,width=9cm]{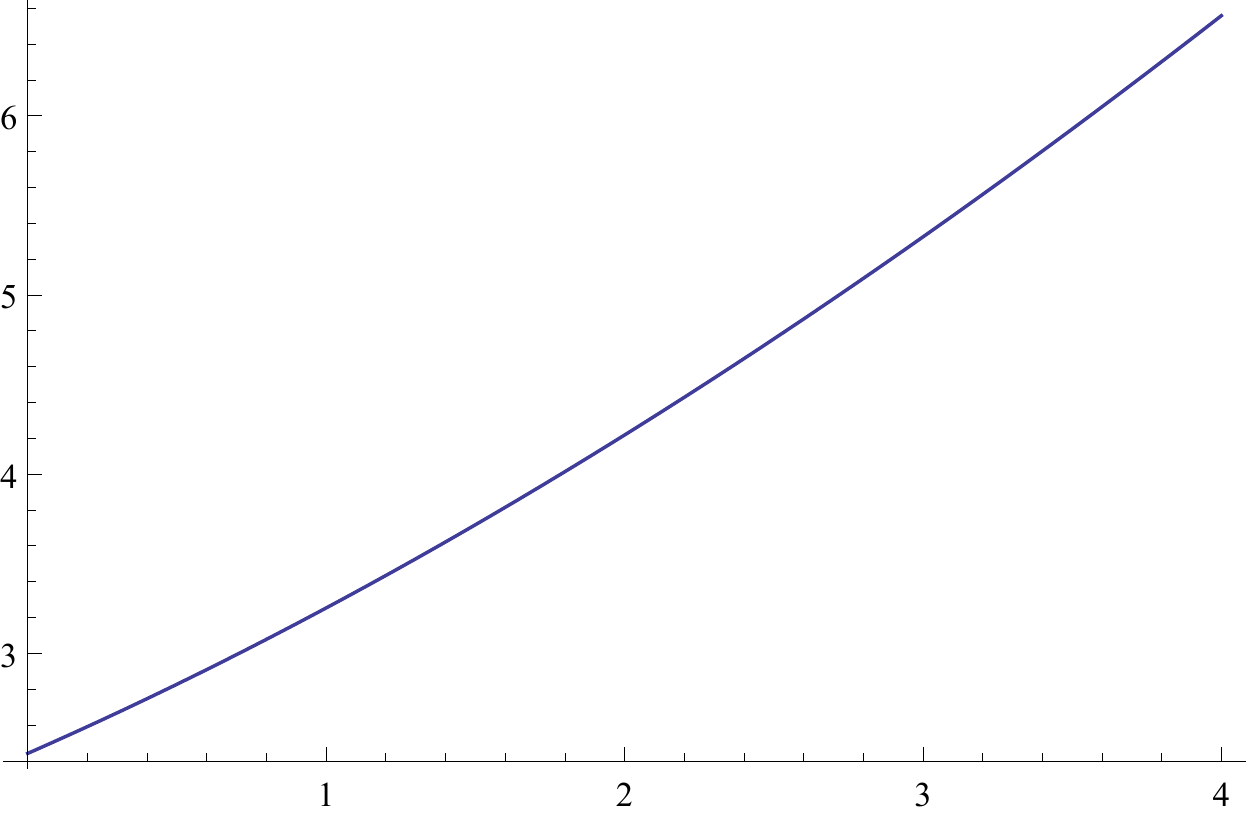}
\caption{Numerical solution of (\ref{odesol})  with $f(0)=\xi=2.44197$, and $f'(0)=0.729925$. The graph was plotted using Mathematica.}   \label{figa} 
\end{figure}

Next, we present some further examples of the methodology discussed in the previous section.

\begin{example}\label{exai}
The Airy function of order 3 solves the following ODE
$$\phi^{(2)}(x)=x\phi(x).$$
Using the same argument as in the proof of Theorem \ref{th1} it follows that 
\begin{eqnarray*}
v^{(2)}(t,x)=xv(t,x)+tv^{(1)}(t,x).
\end{eqnarray*}
This last expression evaluated at $f$ yields
\begin{eqnarray*}
\frac{v^{(2)}}{v^{(1)}}(t,f(t))=t.
\end{eqnarray*}
Finally, from the previous expression and (\ref{hit}) it follows that
\begin{eqnarray*}
f'(t)&=&-\frac{1}{2}\frac{v^{(2)}}{v^{(1)}}(t,f(t))\\
&=&-\frac{1}{2}t.
\end{eqnarray*} 

Thus, for some constant $C$
$$f(t)=C-\frac{1}{4}t^2.$$
See \cite[p.126 ]{vallee}.  For applications of the Airy function in the first hitting time problem of Brownian motion up to a quadratic function see \cite{martin}.
\end{example}
\begin{example} Given the following ODE
\begin{eqnarray*}
\phi^{(2)}(x)=x\phi(x)+\phi^{(1)}(x),
\end{eqnarray*}
and following the same line of reasoning as in the proof of Theorem \ref{th1} we have
\begin{eqnarray*}
v^{(2)}(t,x)=xv(t,x)+tv^{(1)}(t,x)+v^{(1)}(t,x).
\end{eqnarray*}
From (\ref{hit}) we obtain
\begin{eqnarray*}
\frac{v^{(2)}}{v^{(1)}}(t,f(t))=t+1.
\end{eqnarray*}
It follows that
\begin{eqnarray*}
f'(t)=-\frac{t+1}{2}
\end{eqnarray*}
or
\begin{eqnarray*}
f(t)=-\frac{t}{2}-\frac{t^2}{4}+C
\end{eqnarray*}
For example  if $C=-2.58811$ the moving boundary problem of the heat equation associated with $\phi$ is solved.
\end{example}
\begin{example} Given the following Bessel   ODE
\begin{eqnarray*}
\phi''(x)=-\left(\frac{5}{2}-\frac{1}{4}x^2\right)\phi(x).
\end{eqnarray*}
Similar calculations as in the previous examples yield
\begin{eqnarray*}
f(t)=\pm\frac{1}{2}\sqrt{t^2-4}.
\end{eqnarray*}
\end{example}
\begin{example}\label{derivativeairy} The derivative of the Airy function $Ai'(x)$ solves
\begin{eqnarray*}
x\phi''(x)=\phi'(x)+x^2\phi(x),
\end{eqnarray*}
which yields
\begin{eqnarray*}
-tv^{(3)}(t,x)+(t^2-x)v^{(2)}(t,x)+(1+2xt)v^{(1)}(t,x)+(x^2+t)v(t,x)=0.
\end{eqnarray*}
Alternatively, from Example \ref{exai} we also have that
\begin{eqnarray*}
v^{(3)}(t,x)-tv^{(2)}(t,x)-xv^{(1)}(t,x)-2v(t,x)=0.
\end{eqnarray*}
Using  the same arguments as those described in Section \ref{secrayleigh} leads to
\begin{eqnarray}\label{abel}
2f(t)f'(t)+(1+tf(t))=0.
\end{eqnarray}
This is the Abel equation  of the second kind and its solution  can be expressed  in terms of the  Airy function of order 3 $Ai$, as follows:
\begin{eqnarray*}
\frac{1}{2}t Ai\left(\frac{t^2}{4}+f(t)\right)+ Ai'\left(\frac{t^2}{4}+f(t)\right)=0.
\end{eqnarray*}
See Figure \ref{figb} for a numerical example with $f(0)=-1.01879$.
\end{example}
\begin{figure}
\hspace{-.5 cm}\includegraphics[scale=.5,height=7cm,width=9cm]{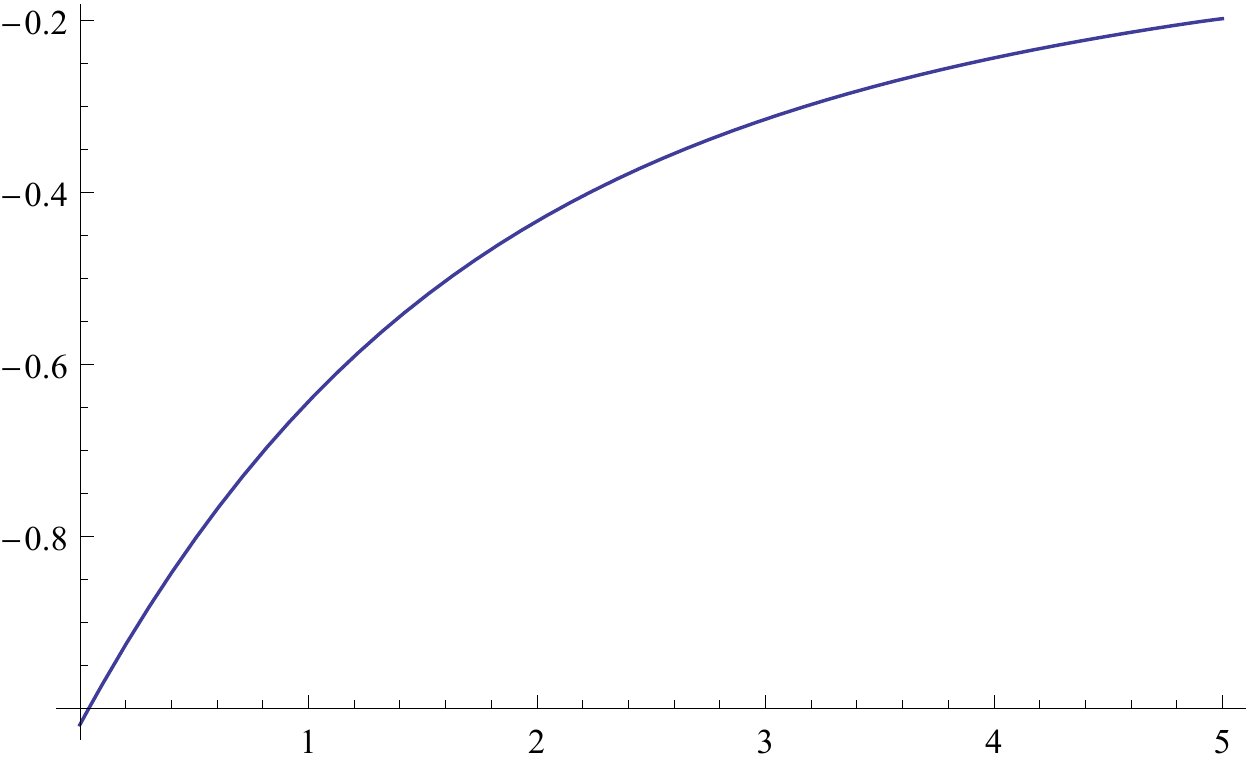}
\caption{The Numerical solution of $f$, defined in Example \ref{derivativeairy}, and such that  $f(0)=-1.01879$ was plotted with Mathematica.}   \label{figb} 
\end{figure}
\section{Zeros of the Pearcey function. An asymptotic approach}\label{secasymptotic}
In this section we carry out an analysis in order to find an asymptotic solution to the moving boundary problem associated with the Pearcey integral. This result is useful when finding the numerical solution of the Rayleigh equation (\ref{odesol}). The main result of this section is the following.
\begin{theorem} Suppose that $v$ is as in (\ref{pearcey}), $\xi$ is any zero of the Airy function of order 3, and 
\begin{eqnarray}\label{aprx}
f(t)=-2(t/3)^{3/2}+\xi(3t)^{1/6},\qquad t\geq 0.
\end{eqnarray}
Then
\begin{eqnarray*}
v(t,f(t))\to 0,
\end{eqnarray*}
as $t\to\infty$.
\end{theorem}
\begin{proof} For brevity let us just consider the term within the brackets in (\ref{pearcey}), {\it i.e.\/},
\begin{eqnarray*}
\exp\left\{i\lambda y-\frac{1}{2}\lambda^2t-\frac{\lambda^4}{4}\right\}.
\end{eqnarray*}
Introduce a variable $-\alpha^2t/2$,
\begin{eqnarray*}
\exp\left\{\lambda(iy+\alpha t)-\frac{1}{2}(\lambda+\alpha)^2t-\frac{\lambda^4}{4}\right\}e^{\frac{1}{2}\alpha^2t}.
\end{eqnarray*}
Set $u=\lambda+\alpha$ and rearrange terms to obtain
\begin{eqnarray*}
&&\exp\left\{u(iy+\alpha t)-\frac{1}{2}u^2t-\frac{(u-\alpha)^4}{4}\right\}e^{-\frac{1}{2}\alpha^2t-i\alpha y}\\
&&\quad=\exp\left\{u(iy+\alpha t+\alpha^3)-\frac{1}{2}u^2(t+3\alpha^2)+\alpha u^3-\frac{u^4}{4}\right\}\\
&&\quad\quad\enskip e^{-\frac{\alpha^4}{4}-\frac{1}{2}\alpha^2t-i\alpha y}.
\end{eqnarray*}
To get rid of the {\it heat\/} (or quadratic) term note that
\begin{eqnarray}\label{alpha}
t+3\alpha^2=0,\qquad \hbox{gives}\qquad \alpha_\pm=\pm i\sqrt{\frac{t}{3}}.
\end{eqnarray}
That is,
\begin{eqnarray*}
\exp\left\{ui\left[y\pm\frac{2}{3^{3/2}}t^{3/2}\right]\pm\frac{i}{3^{1/2}}t^{1/2} u^3-\frac{u^4}{4}\right\}e^{-\frac{\alpha^4}{4}-\frac{1}{2}\alpha^2t-i\alpha y}.
\end{eqnarray*}
Next,  if we choose $\alpha_+$, as in (\ref{alpha}),
\begin{eqnarray*}
\exp\left\{ui\left[y+\frac{2}{3^{3/2}}t^{3/2}\right]+\frac{i}{3^{1/2}}t^{1/2} u^3-\frac{u^4}{4}\right\}e^{\frac{5}{36}t^2+\sqrt{\frac{t}{3}}y},
\end{eqnarray*}
and thus, from (\ref{pearcey}),
\begin{eqnarray*}
v(t,y)=e^{\frac{5}{36}t^2+\sqrt{\frac{t}{3}}y}\frac{1}{2\pi}\int_{-\infty}^\infty e^{ui\left[y+\frac{2}{3^{3/2}}t^{3/2}\right]+i\sqrt{3t} \frac{u^3}{3}-\frac{u^4}{4}}du.
\end{eqnarray*}
Now, let $z=u(3t)^{1/6}$, $u=z(3t)^{-1/6}$, and $(3t)^{-1/6}dz=du$, which yields
\begin{eqnarray*}
(3t)^{-1/6}e^{\frac{5}{36}t^2+\sqrt{\frac{t}{3}}y}\frac{1}{2\pi}\int_{-\infty}^\infty e^{iz\frac{\left[y+\frac{2}{3^{3/2}}t^{3/2}\right]}{(3t)^{1/6}}+i \frac{z^3}{3}-\frac{z^4}{4(3t)^{2/3}}}dz,
\end{eqnarray*}
or equivalently 
\begin{eqnarray*}
(3t)^{1/6}e^{-\frac{5}{36}t^2-\sqrt{\frac{t}{3}}y}v(t,y)=\frac{1}{2\pi}\int_{-\infty}^\infty e^{iz\frac{\left[y+\frac{2}{3^{3/2}}t^{3/2}\right]}{(3t)^{1/6}}+i \frac{z^3}{3}-\frac{z^4}{4(3t)^{2/3}}}dz.
\end{eqnarray*}
Letting $y=-2(t/3)^{3/2}+\xi(3t)^{1/6}$
\begin{eqnarray*}
(3t)^{1/6}e^{\frac{t^2}{12}-\xi\frac{t^{2/3}}{3^{1/3}}}v\left(t,-2\left[\frac{t}{3}\right]^{3/2}+\xi(3t)^{1/6}\right)=\frac{1}{2\pi}\int_{-\infty}^\infty e^{iz\xi+i \frac{z^3}{3}-\frac{z^4}{4(3t)^{2/3}}}dz.
\end{eqnarray*}
Hence for arbitrary $\xi$
\begin{eqnarray*}
\lim\limits_{t\to\infty}\left\{(3t)^{1/6}e^{\frac{t^2}{12}-\xi\frac{t^{2/3}}{3^{1/3}}}v\left(t,-2\left[\frac{t}{3}\right]^{3/2}+\xi(3t)^{1/6}\right)\right\}=\hbox{Ai}(\xi).
\end{eqnarray*}
In particular if $\xi$ is a zero of the Airy function  then
\begin{eqnarray*}
\lim\limits_{t\to\infty}\left\{(3t)^{1/6}e^{\frac{t^2}{12}-\xi\frac{t^{2/3}}{3^{1/3}}}v\left(t,-2\left[\frac{t}{3}\right]^{3/2}+\xi(3t)^{1/6}\right)\right\}&=&\hbox{Ai}(\xi)\\
&=&0.
\end{eqnarray*}
\end{proof}
\begin{examplen}
For instance, if $\xi=-2.33811$ we have that $f(t)=2(t/3)^{3/2}+2.33811(3t)^{1/6}$. See Figure \ref{fig8}.
\end{examplen}
Next, using the results of this Section and Section \ref{secrayleigh}, we present an algorithm which can be used to solve  (\ref{odesol}) in Theorem \ref{th1} . 
\begin{algorithm} For some root $\xi$ of the Airy function Ai, let $\tilde{f}$ be defined as in (\ref{aprx}).  Given an arbitrary time $t\in\mathbb{R}^+$ we may find a solution to equation (\ref{odesol}) in the interval $[t-\varepsilon,t+\varepsilon]$ as follows.
\begin{eqnarray*}
\left\{ \begin{array}{l}
\hbox{Given $t$ and $\xi$,  find a root of function $v$ , defined in (\ref{pearcey}), starting at $\tilde{f}(t)$.}\\
\hbox{Lef $f(t)$ be the root obtained in the previous step,}\\
\quad\hbox{ then $f'(t)=-1/2\cdot v^{(2)}(t,f(t))/v^{(1)}(t,f(t))$.}\\
\hbox{With $f(t)$ and $f'(t)$ solve (\ref{odesol}) in $[t-\varepsilon,t+\varepsilon]$, for some $\varepsilon>0$.}
\end{array}
  \right.
\end{eqnarray*}
\begin{examplen} Suppose we choose $t=10$ and $\varepsilon=2$. Then the procedure is the following
\end{examplen}
{\tiny
\begin{lstlisting}
F1[t_, x_] := 
 1/(2*Pi)*NIntegrate[(I*y)*
    Exp[I*x*y - y^2/2*t - y^4/4], {y, -Infinity, Infinity}]
F2[t_, x_] := 
 1/(2*Pi)*NIntegrate[(I*y)^2*
    Exp[I*x*y - y^2/2*t - y^4/4], {y, -Infinity, Infinity}]
asy[t_] := 2*(t/3)^(3/2) - AiryAiZero[1]*(3*t)^(1/6)
x10 = FindRoot[Re[F0[10, t]], {t, N[asy[10]]}]
x10 = t /. x10
fp10 = -Re[F2[10, x10]]/(2*Re[F1[10, x10]])
s10 = NDSolve[{g''[x] == -1/4*g[x] + 2*(g'[x])^3 - 1/2*x*g'[x], 
   g[10] == x10, g'[10] == fp10}, g, {x, 8, 12}, AccuracyGoal -> 20, 
  PrecisionGoal -> 10, WorkingPrecision -> 30]
\end{lstlisting}
}

\end{algorithm}
\begin{figure}
\hspace{-.5 cm}\includegraphics[scale=.5,height=7cm,width=9cm]{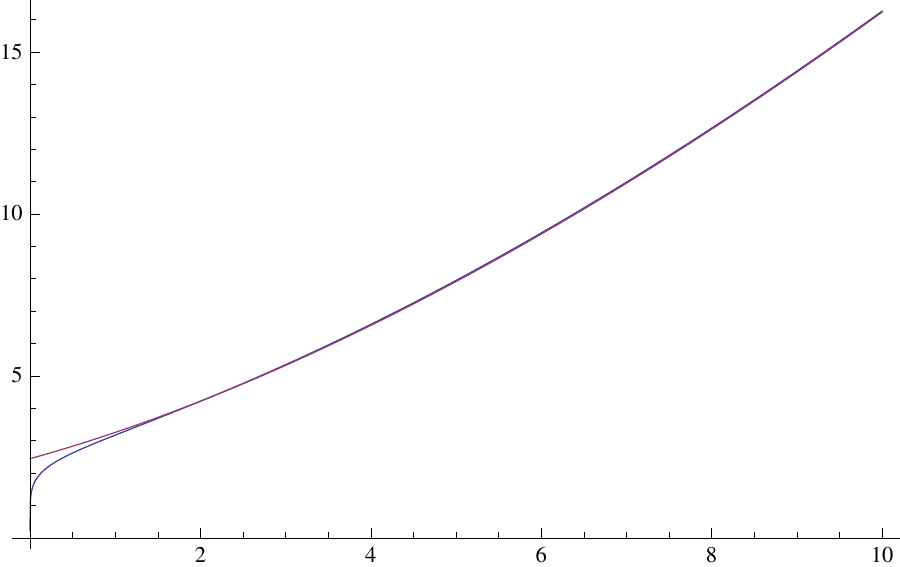}
\caption{The red line is the real boundary $f$, the blue is as in ($\ref{aprx}$), with $\xi=-2.33811$. The graph was plotted with Mathematica.}   \label{fig8} 
\end{figure}
\section{Possible applications and work in progress}
Due to the stochastic and periodic nature of several economic variables, as for instance Mexico's general CPI or the Fruit and Vegetable annual inflation  and assuming $W$ is a random walk, these processes can be modelled as
\begin{eqnarray}\label{sine}
X_t=\sum\limits_{j=1}^n\beta_j\sin(\phi_j+2\pi\nu_jt)+W_t,\qquad t=1,2,\dots,
\end{eqnarray}
where the $\beta_j$ and $\phi_j$ represent  respectively the amplitude and phase at a time given frequency $\nu_j$. In turn a continuous time approximation of (\ref{sine}) can be expressed in terms of  the solution of an SDE of the form since
\begin{eqnarray}\label{sde}
dX_t&=&\mu(t,X_t)dt+\sigma(t,X_t)dW_t,\\
\nonumber &\sim&\\
\nonumber\Delta X_j&=&\mu(j,X_j)\Delta j+\sigma(j,X_j)\Delta W_j,\qquad j=0,1,\dots,
\end{eqnarray}
where functions $\mu$ and $\sigma$ are respectively:
\begin{eqnarray*}
\mu(t,X_t)=\sum\limits_{j=1}^n2\pi\nu_j\beta_j\cos(\phi_j+2\pi\nu_jt),\quad\hbox{and}\quad\sigma(t,X_t)=1,
\end{eqnarray*}
and $W$ is a Wiener process. A reasonable set of questions that could be asked could be for instance:
\begin{eqnarray*}
\hbox{What is the probability that Fruit and Vegetable annual}\\
\hbox{ inflation will reach 20 points before the end of 2016?}\\[0.3cm]
\hbox{What is probability that the general CPI will remain}\\
\hbox{between 3 and 4 percent until the end of 2017?} 
\end{eqnarray*}
As it turns out, to  answer the previous questions it is necessary to understand the moving boundary problem of heat equation addressed in this work. More specific examples is still work in progress.

\section{Concluding remarks}\label{conclusions}
In this work we find the zeros of the Pearcey function, in terms of the solution of a Rayleigh-type equation. This  goal is achieved by exploiting, on the one hand, the differential equation of an Airy function of order 4 and on the other by using the fact that the Pearcey function is a solution of the heat equation. As a by-product we develop a methodology, using straightforward techniques, to solve the moving boundary problem of the heat equation in the case in which the convolving function is a generalized Airy function. We expect that the techniques described within can be used in the construction of densities of the first hitting time problem of Brownian motion. The scope and applicability to the latter problem is still work in progress.

\end{document}